\documentclass[11pt,a4paper,oneside,reqno]{amsart}

\usepackage{graphics}
\usepackage{color,graphicx}
\usepackage{bm} 
\usepackage{amsfonts,amssymb,amsmath}

\newcommand{\oneandhalfspace}{\renewcommand{\@defaultbaselinestretch}{1.1}}
\newtheorem{theorem}{Theorem}
\newtheorem{lemma}[theorem]{Lemma}
\newtheorem{proposition}[theorem]{Proposition}

\def\dist{\mathrm{dist}}
\def\dd{\mathrm{d}}
\def\dw{W}

\def\P{\mathcal{P}}
\newcommand{\R}{\mathbb{R}}
\def\LL{\mathcal{L}}
\def\TT{\mathcal{T}}
\def\T{\mathbb T}
\def\M{\mathcal M}
\def\Z{\mathbb{Z}}
\let\e\varepsilon

\def\Eper{E^{\mathrm{per}}}
\DeclareMathOperator{\supp}{supp}
\DeclareMathOperator*{\essinf}{ess\,inf}
\DeclareMathOperator*{\esssup}{ess\,sup}
\def\eps{\varepsilon}

\def\XXint#1#2#3{{\setbox0=\hbox{$#1{#2#3}{\int}$}
     \vcenter{\hbox{$#2#3$}}\kern-.5\wd0}}

\def\c#1{c_{#1}}

\begin{document}
\title[Optimality of the Triangular Lattice]{Optimality of the Triangular Lattice for a Particle System with Wasserstein Interaction}

\author{D.P.~Bourne$^1$}
\author{M.A.~Peletier$^2$}
\author{F.~Theil$^3$}

\maketitle

\begin{abstract}
We prove strong crystallization results in two dimensions for an energy that arises in the theory of block copolymers. The energy is defined on sets of points and their weights, or equivalently on the set of atomic measures. It consists of two terms; the first term is the sum of the square root of the weights, and the second is the quadratic optimal transport cost between the atomic measure and the Lebesgue measure.

We prove that this system admits crystallization in several different ways: (1) the energy is bounded from below by the energy of a 
triangular lattice (called $\mathcal T$); (2) if the energy equals that of $\mathcal T$, then the measure is a rotated and translated copy of $\mathcal T$; (3)
if the energy is close to that of $\mathcal T$, then locally the measure is close to a rotated and translated copy of $\mathcal T$.
These three results require the domain to be a polygon with at most six sides. A fourth result states that the energy of $\mathcal T$ can be achieved in the limit of large domains, for domains with arbitrary boundaries.

The proofs make use of three ingredients. First, the optimal transport cost associates to each point a polygonal \emph{cell}; the energy can be bounded from below by a sum over all cells of a function that depends only on the cell. Second, this function has a convex lower bound that is sharp at $\mathcal T$. Third, Euler's polytope formula limits the average number of sides of the polygonal cells to six, where six is the number corresponding to the triangular lattice.
\end{abstract}
\section{Introduction}

\footnotetext[1]{School of Mathematics and Statistics,
University of Glasgow,
15 University Gardens,
Glasgow G12 8QW, UK.}

\footnotetext[2]{Department of Mathematics and Computer Science and Institute for Complex Molecular Systems, Technische Universiteit Eindhoven,
PO Box 513, 5600 MB Eindhoven,
The Netherlands.}

\footnotetext[3]{Mathematics Institute,
Zeeman Building,
University of Warwick,
Coventry CV4 7AL, UK.}

\subsection{The setting}
Many materials achieve their state of lowest energy with a periodic arrangement of the atoms: their ground states are \emph{crystalline}. Many other systems also favor ordered, periodic structures; examples are packed spheres~\cite{HalesHarrisonMcLaughlinNipkowObuaZumkeller10}, convection cells (e.g.~\cite{KoschmiederPallas74}), reaction-diffusion systems~\cite{Hoyle06}, higher-order variational systems (e.g.~\cite{LloydSandstedeAvitabileChampneys08}) and also block copolymers, the system that inspired the energy that we study in this paper. On the other hand, there are also many examples of deviation from periodicity: entropy may overrule order, defects may appear, and even non-periodic ground states exist, as in the case of quasicrystals~\cite{Barber09}.

It follows that the question whether and why a given system favors periodicity is a non-trivial one. It is also an important one, since many material properties depend strongly on the microscopic arrangement of atoms or particles. And it is a surprisingly hard question to answer.

In two and three dimensions, the strongest results are available for two-point interaction energies of the form $\sum_{i\ne j} V(x_i-x_j)$. In the case of hard-sphere repulsion the triangular arrangement in two dimensions is easily recognized as optimal, but the highest-density stacking of spheres in three dimensions was computed by Hales in 1998 in a proof that is still being formalized~\cite{HalesHarrisonMcLaughlinNipkowObuaZumkeller10}. For various Lennard-Jones-like interaction potentials $V$ with sufficiently short range it has been proved that global minimizers in two dimensions  are triangular under appropriate boundary conditions~\cite{Radin81,Theil06,YeungFrieseckeSchmidt12}. E and Li show that addition of suitable three-point interactions shifts the ground state from the triangular to a \emph{hexagonal} lattice~\cite{ELi09}.

\medskip

For systems with more general interactions between the particles, however, we know of no rigorous results; in this paper we study a system in this class, and prove several strong crystallization results.

Let $\Omega\subset \R^2$ be fixed such that $|\Omega|=1$.
The system is described by a finite number of \emph{points}~$z_i$ in $\Omega$ and their \emph{masses} $v_i$, or equivalently by an \emph{atomic measure}, i.e., a positive measure $\mu$ of the form
\begin{equation}
\label{def:mu}
\mu = \sum_{z\in Z} v_z \delta_z,  \qquad\text{with } v_z> 0 \text{ and }\sum_{z\in Z}v_z = 1,
\end{equation}
where $Z$ is any finite subset of $\Omega$.
The (unscaled) energy of the system is
\[
 \hat E_\lambda(\mu)  =\lambda \sum_{z \in Z}
 \mu\left(\{z\}\right)^\frac{1}{2}  + \dw(\LL_\Omega, \mu).
\]
Here $\LL_\Omega$ is the Lebesgue measure on $\Omega$.
The function $\dw$ is the quadratic optimal transport cost; see~\cite{Villani03} for an extensive introduction to this topic. For our purposes it is sufficient to define $W(\LL_\Omega, \mu)$ for $\mu$ of the form~\eqref{def:mu}:
\begin{equation}
\label{def:W}
\dw(\LL_\Omega, \mu) := \inf \left\{  \int_\Omega |x-T(x)|^2\, \dd x : \;
  T\colon\Omega \to Z, \, |T^{-1}(z)| = \mu(\{z\}) \, \forall z \right\}.
\end{equation}
By, e.g.,~\cite[Theorem 2.12]{Villani03} there exists an optimal map $T$.

\medskip
This system arises as a highly stylized model for block copolymer melts. The copolymers consist of two parts, called the A and B parts; the A and B parts strongly repel each other, leading to phase separation, but since they are connected to each other by a covalent bond, the phases have to be microscopically mixed. In the regime described here, the B parts have much larger volume than the A parts, and therefore the A parts congregate into small balls represented by the points $z$; the B parts fill the remaining volume. The masses $v_z = \mu(\{z\})$ are the relative amount of A at the point $z$.

The two terms in $\hat E_\lambda$ represent the two important contributions to the energy. The first term measures the (rescaled) interfacial area separating the two phases; since the A phase resembles a small ball of volume $v_z$, its interfacial area is proportional to $v_z^{1/2}$. The second term is an energetic penalty for a large separation between the A and B parts: the map $T$ maps a B particle to its corresponding A particle, and $|x-T(x)|^2$ measures the energy of the covalent bond (modeled by a linear spring) connecting the two particles. We discuss the modeling background of this system in more detail in~\cite{BournePeletierRoper}.

\medskip
This system has a number of distinguishing features.
\begin{enumerate}
\item It is a system of `particles' that interact with each other via the nonlocal functional~$W$. This nonlocal functional potentially allows each particle to interact with all other particles simultaneously. This makes it different from particle systems with two-, three-, or four-particle interactions.
\item Each particle carries a `weight' $\mu(\{z\})$ that influences the interaction.
\item There is no imposed length scale: the length scale is determined in the competition between the two terms, much as in the case of other block copolymer models~\cite{AlbertiChoksiOtto09,Choksi01,ChoksiPeletierWilliams09,Muratov10}.
\item The number of particles is not fixed in advance: it also arises from the trade-off between the terms.
\end{enumerate}

We will see below that for minimizers the number of particles scales approximately as $\lambda^{-2/3}$. In the limit $\lambda\to0$, therefore, the typical number of particles for a minimizer becomes unbounded. Numerical calculations suggest that in this limit the particles organize themselves in a regular triangular pattern, as illustrated in Figure~\ref{fig:NumericsSmallLambda}.
The aim of this paper is to characterize and prove this phenomenon of crystallization.
\begin{figure}[ht]
\begin{center}
\includegraphics[width=0.75\textwidth]{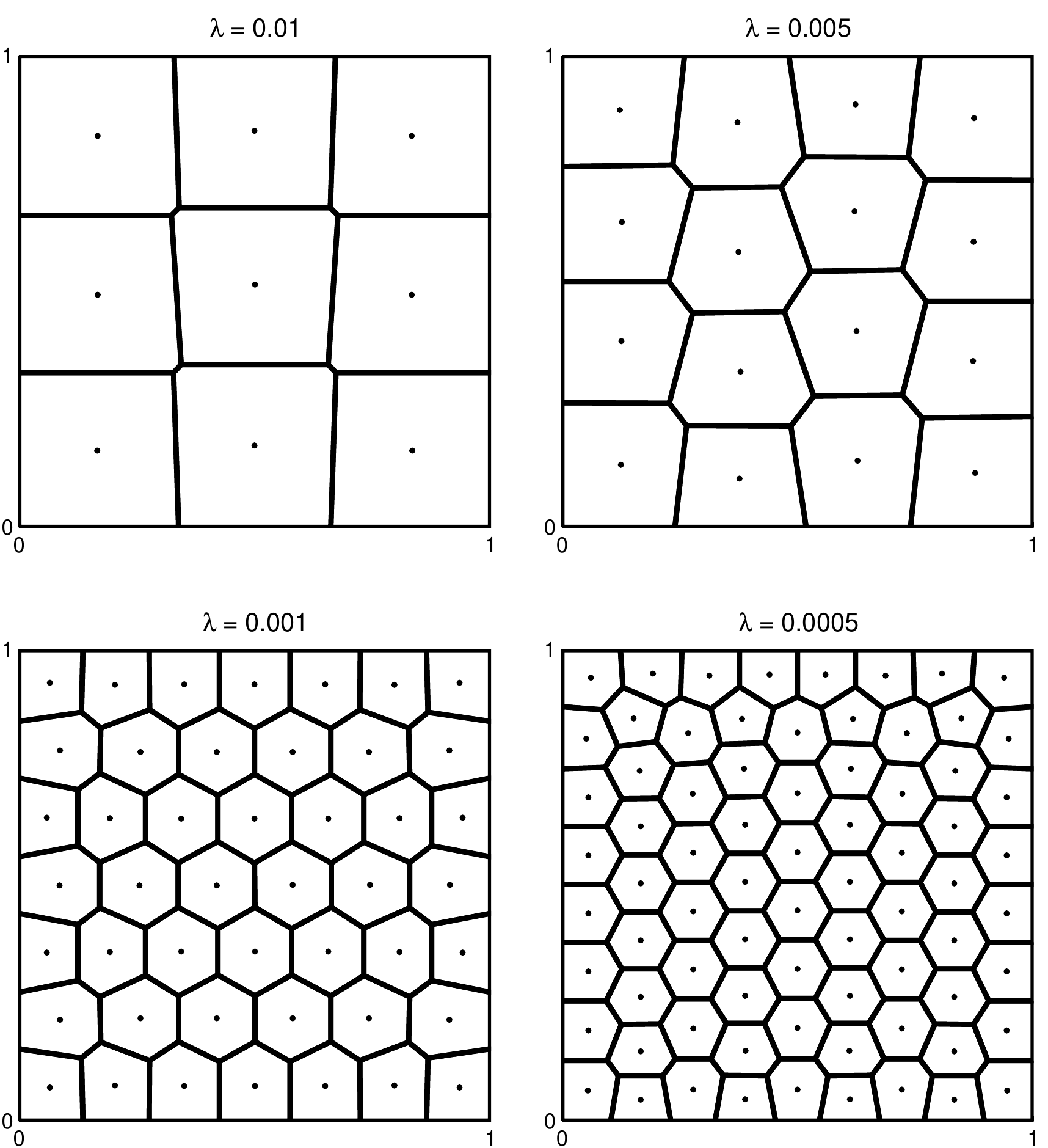}
\caption{\label{fig:NumericsSmallLambda} Minimizers of $\hat{E}_\lambda$. As $\lambda$ decreases, the optimal number of points $z\in Z$ increases as $\lambda^{-2/3}$, and they organize in a nearly-triangular lattice. The polygons surrounding the points are the \emph{cells} $T^{-1}(Z)$ (see Section 2), and they approximate regular hexagons as $\lambda\to0$. The numerical method used to obtain this figure is described in the companion paper~\cite{BournePeletierRoper}.}
\end{center}
\end{figure}

To be concrete we prove four results that each characterizes the phenomenon of crystallization in a different way. We assume that $\Omega$ is a polygon with at most six sides.
\begin{enumerate}
\item \emph{An energy bound:} We show that for any $\lambda>0$ the energy of an arbitrary configuration is bounded from below by the energy of an optimal triangular lattice (Theorem~\ref{th:lowerbound}).
\item \emph{The energy bound is sharp:} In the limit $\lambda\to0$ this bound can be obtained; or equivalently, for fixed $\lambda$, the bound can be reached in the limit of large domains (Theorem~\ref{th:limit}).
\item \emph{Exact crystallization:} If the energy bound is achieved \emph{exactly}, then the structure is exactly triangular with the optimal separation between the points (Theorem~\ref{th:stability}).
\item \emph{Geometric stability:} If the energy bound is not exactly achieved, but the gap in the bound is small for small $\lambda>0$, then the structure is asymptotically triangular (Theorem~\ref{th:stability}).
\end{enumerate}
Some of these results also hold for other domains.
For the precise statement of these results we first introduce some notation.
\subsection{Setting up the results: Rescaled energy}
We  scale space in such a way that small $\lambda$ and large domains become the same thing.  The new domain will have (two-dimensional) volume
\begin{equation}
\label{def:c6}
V_\lambda := \left(\frac{2c_6}{\lambda}\right)^\frac{2}{3}, \qquad\text{where }c_6 = \frac{5\sqrt{3}}{54}= 0.160375 \ldots
\end{equation}
The constant $c_6$ is central in this work, and we will comment on it later. For fixed $\Omega\subset\R^2$ with $|\Omega|=1$ we therefore define the scaled domain
\begin{equation}
\label{def:Omegalambda}
\Omega_\lambda := V_\lambda^{1/2}\Omega,
\end{equation}
and for given $\mu_0\in \P(\Omega)$ we define a rescaled measure $\mu\in \M_{\geq0} (\Omega_\lambda)$ by $\mu (A) := V_\lambda \mu_0(V_\lambda^{-1/2}A)$ for any Borel set $A$.
Under this rescaling the energy $\hat E_\lambda$ becomes, up to a factor $\lambda^{4/3}(2c_6)^{-4/3}$, $E_\lambda : \M_{\geq 0}(\Omega_\lambda)\to \R$,
\begin{equation*}
E_\lambda (\mu) :=
 2c_6\,\sum_{z \in Z}\mu(\{z\})^\frac{1}{2}+ \dw(\LL_{\Omega_\lambda},\mu)
\end{equation*}
provided $\mu$ is atomic and $\mu(\Omega_\lambda) = |\Omega_\lambda| = V_\lambda$, and $E_\lambda (\mu):=\infty$ otherwise.
This is the energy that we shall consider throughout this paper.
In this scaling, we expect $\mu$ to consist of $O(V_\lambda)$ points, each with $O(1)$ mass, and spaced at distance $O(1)$. The crystallization results below are a much stronger version of this statement.

\subsection{Results}

Throughout the rest of this paper, the energy functionals will always be defined with respect to a set $\Omega_\lambda$ which is constructed as in~\eqref{def:Omegalambda} out of a unit-area set $\Omega$ and a parameter $\lambda>0$.

\begin{theorem}
\label{th:lowerbound}
Let $\Omega$ be a polygon with at most six sides with $|\Omega|=1$. Then for all $\lambda > 0$ we have the lower bound
\begin{equation}
E_\lambda \geq 3\,c_6 V_\lambda \label{ceq1}.
\end{equation}
\end{theorem}
As we show below, the right-hand side in~\eqref{ceq1} is the energy of a structure of $V_\lambda$ regular hexagons of area $1$. This lower bound can also be achieved, and this is even possible for more general sets~$\Omega$:

\begin{theorem}
\label{th:limit}
Let $\Omega$ be a bounded, connected domain in $\mathbb{R}^2$ with $|\Omega|=1$ and such that $\partial \Omega=\varphi(K)$ for some Lipschitz function $\varphi:\mathbb{R}\to \mathbb{R}^2$ and some compact set $K \subset \mathbb{R}$.
Then
\begin{align}
\lim_{\lambda \to 0}  V_\lambda^{-1}\,\inf
E_{\lambda}
\;=\; 3\,c_6. \label{ceq2}
\end{align}
\end{theorem}
The triangular lattice $\TT$ with density 1 is defined as
\begin{equation}
\label{def:T}
\TT = \left\{ \frac1{12^{1/4}}\left({{2}\atop{0}} {{1}\atop{\sqrt{3}}}\right)k \; : \; k \in \Z^2\right\}\subset \R^2.
\end{equation}
The normalising factor $12^{-1/4}$ is introduced so that
 the points of $\mathcal{T}$ are the centres of regular unit area hexagons tiling the plane.
When the inequality~\eqref{ceq1} is saturated or nearly saturated, then the structure is exactly triangular or nearly so:
\begin{theorem}
\label{th:stability}
Assume the conditions of Theorem~\ref{th:lowerbound}.
\begin{itemize}
\item[(a)]
If $E_\lambda(\mu) = 3c_6V_\lambda$, then $\mu$ is an atomic measure with all weights equal to $1$, and $\supp\mu$ a translated and rotated copy of the triangular lattice $\TT$.
\item[(b)] Define the dimensionless \emph{defect} of a measure $\mu$ on $\Omega_\lambda$ as
\[
d(\mu) := V_\lambda^{-1} E_\lambda(\mu) - 3c_6.
\]
There exists $C>0$ such that for $\lambda<C^{-1}$ and for all $\mu$ with $d := d(\mu)\leq C^{-1}$, $\supp \mu$ is $O(d^{1/6})$ close to a triangular lattice, in the following sense: after eliminating $C V_\lambda d^{1/6}$ points, the remaining points have six neighbors whose distance lies between $(1-Cd^{1/6})$ and $(1+Cd^{1/6})$ of the optimal distance $2^{1/2}3^{-3/4}$.

\end{itemize}
\end{theorem}

Part (a) of Theorem~\ref{th:stability} is a natural counterpart of part (b), which should apply when $d(\mu)=0$. In this case, since $\Omega$ is a polygon with at most six sides, this assertion is nearly empty: the only domain $\Omega_\lambda$ for which equality can be achieved is the case when $\Omega$ is a regular hexagon of area $1$ and $\lambda=1$, and $\Omega_\lambda\cap \TT$ is a single point at the origin.

However, the methods of this paper can be extended to the case of `periodic domains', and we give an example here. Let us define $\T$ to be a rectangle $[0,\gamma)\times [0,\gamma^{-1})$ with area one and periodic boundary conditions, or more precisely, as the two-dimensional torus
\[
\T = \R^2/(\gamma\Z\times \gamma^{-1}\Z).
\]
As before, $\T_\lambda$ is the blown-up version of $\T$, and the energy $E_\lambda$ has a natural analogue $\Eper_\lambda$ on $\T_\lambda$.
\begin{theorem}
\label{th:periodic}
If $\Eper_\lambda(\mu) = 3c_6V_\lambda$, then $\mu$ is an atomic measure with all weights equal to $1$, and $\supp\mu$ a translated and rotated subset of the triangular lattice $\TT$.
\end{theorem}
Naturally, equality can only be achieved if the size and aspect ratio of $\T_\lambda$ are commensurate with the periodicity of the triangular lattice.

\subsection{Discussion}

In this section we comment on a number of similarities and differences with other results.

\emph{Exact and approximate crystallization.}
In the introduction we mentioned the results of Radin, Theil, and Yeung-Friesecke-Schmidt~\cite{Radin81,Theil06,YeungFrieseckeSchmidt12} on exact crystallization for systems of points in the plane. In one dimension there are many more results that prove that minimizers of some functional are exactly periodic; examples are the block copolymer-inspired systems studied by M\"uller~\cite{Muller93} and Ren and Wei~\cite{RenWei00},  the Swift-Hohenberg energy~\cite{PeletierTroy96}, and two-point interaction systems of the form $\sum_{i,j:i\not=j} V(x_i-x_j)$~\cite{Ventevogel78,VentevogelNijboer79}.

An important class of related functionals in two dimensions arises from the `location problem' or `optimal configurations of points' (see, e.g., \cite{BouchitteJimenezMahadevan} and \cite{ButtazzoSantambrogio}). An example of such a problem is
\begin{equation}
\label{minpb:location}
\inf_{\substack{Z \subset \Omega \\ |Z|=n}}  G_\Omega(Z), \quad \textrm{where} \quad  G_\Omega(Z) = \int_\Omega \big[\min_{z\in Z} f(|x-z|)\big]\, \dd x.
\end{equation}
Here $f:[0,\infty)\to\R$ is a given non-decreasing function and $n$ is given.
The set $Z$ is finite and $|Z|$ denotes the counting measure of $Z$.
When $f(r) = r^2$, then this problem is in fact identical to
\[
\inf_{\bm \chi \in \mathcal{C}(\Omega)} \sum_i I_\Omega(\chi_i)
\]
in the notation of Section~\ref{sec:preliminaries} (see equations \eqref{def:C} and \eqref{def:I}).
For problem \eqref{minpb:location}, a variety of different crystallization results exist. L. Fejes T\'oth showed that if the domain $\Omega$ is a polygon with at most six sides, the expression~\eqref{minpb:location} is bounded from below by  $n$ times the same expression calculated for a regular hexagon of area $|\Omega|/n$ (Theorem~\ref{thm:wasslatt} below is a version of this; see~\cite{Fejes-Toth72,Gruber99,MorganBolton02}). G.~Fejes T\'oth gave an improved version that includes a stability statement~\cite{G-Fejes-Toth01}, which we include below as Lemma~\ref{lem:GFT}.
Although `optimality' in this location problem is defined differently than optimality for the energy $E_\lambda$ of this paper, the two `energies' are close enough to allow the results by the two Fejes T\'oth's to be applied to the structures of this paper. These results therefore figure centrally in the arguments below.

\emph{Boundaries have positive energy. }One interpretation of Theorems ~\ref{th:lowerbound} and ~\ref{th:limit} is that an imperfect boundary contributes a \emph{positive} energy to the system, provided it does not have too many sides; on the other hand, as we shall see below, curved boundaries can actually be better than polygonal ones. This boundary penalization is similar to the case of Lennard-Jones-type potentials, but different from the case of fully repulsive potentials.

\emph{Neighbors and the connectivity graph. }For Lennard-Jones systems one often defines `neighbors' of a point $x_i$ as those points $x_j$ such that $V(x_i-x_j)$ is close to $\min V$. Although this definition contains an arbitrary choice of `closeness', it works well because flat geometry creates hard limits on how many neighbors there may be (six in two dimensions, twelve in three). In the system of this paper, no such limit exists; a point can have an arbitrarily large number of neighbors, and indeed this is energetically favorable \emph{for that point} (but not for the others), as we shall see below.

Instead of a local limit on the number of neighbors, there is a global limit of a graph-theoretic nature: Euler's polyhedral formula limits the \emph{average} number of neighbors to six. For this property to hold, the boundary~$\partial\Omega$ should not introduce too many vertices and sides, and this is the origin of the restriction in Theorems~\ref{th:lowerbound} and~\ref{th:stability} on the number of sides of~$\Omega$.

\emph{The Abrikosov lattice. }Our problem and the location problem also have strong links with vortex lattice problems like the Abrikosov lattice, which is
observed in superconducting materials. More precisely, we say that a function $h$ is in the admissible class $\mathcal A^2$ if
\begin{align}  \label{EL}
-\Delta h = \mu - 1 \text{ in } \R^2,
\end{align}
for some positive measure $\mu$ of the form
$$ \mu = \sum_{z \in Z} \delta_z,$$
where $Z \subset \R^2$ is countable. In \cite{SandierSerfaty2012a} a renormalized Coulomb energy $\hat S_\Omega^2(h)$ is associated to $h$ and a domain $\Omega \subset \R^2$. The renormalized energy has the property that \begin{align*}
\lim_{\eps \to 0}\left(\|\nabla h_1^\eps\|_{L^2(\Omega)}^2 - \|\nabla h_2^\eps\|_{L^2(\Omega)}^2\right) = \hat S^2_\Omega(h_1)-\hat S^2_\Omega(h_2),
\end{align*}
if $h_1,h_2$ are admissible, $h^\eps$ is a mollification of $h$ and $\# (Z_1 \cap \Omega) = \#(Z_2 \cap \Omega)$.

It is conjectured in \cite{SandierSerfaty2012a} that the renormalized energy density
$$\lim_{R \to \infty} R^{-2} \hat S_{R \Omega}^2$$
is minimized by a triangular lattice ($Z = \TT$), which is interpreted as the Abrikosov lattice in the context of superconductivity. This conjecture admits a natural generalization where (\ref{EL}) is replaced by the $p$-Laplacian (defined by $\Delta_p h  = \mathrm{div}(|\nabla h|^{p-2} \nabla h)$). We say that for an atomic measure $\mu$
with $\mu(\Omega)= |\Omega|$
the function $h \in W^{1,p}(\Omega)$, $2<p< \infty$, is in the admissibility class $\mathcal A^p(\mu)$ if
$$ -\Delta_p h = \mu-1 \text{ in } \Omega, \quad \frac{\partial h}{\partial \nu} =0 \text{ on } \partial \Omega . $$
We say that $h \in W^{1,\infty}(\Omega)$ is in the admissibility class $\mathcal A^\infty(\mu)$ if
there exists a map $T: \Omega \to \supp \mu$ such that $|T^{-1}(\{z\})| = \mu(\{z\})$ for $z \in \supp \mu$ and
$$h(x) = |x-T(x)|.$$
The energy of $h\in \mathcal A^p$ is defined by
$$S^p_\Omega(h)=
\left\{\begin{array}{ll}
\displaystyle \frac{1}{p'}\|\nabla h\|_{L^p(\Omega)}^p &
\text{ if } 2<p< \infty,\\
\displaystyle \int_\Omega h(x)\, \dd x & \text{ if } p = \infty.
\end{array} \right.$$

We conjecture that the infimum of $\lim_{R\to \infty} \min_{h \in \mathcal A^p}R^{-2} S^p_{R\Omega}(\cdot)$
is realized by the triangular lattice $\mathcal T$, i.e., if $\mu_{\mathcal T}= \sum_{z \in \mathcal T}\delta_z$, then
\begin{multline}
\nonumber
\lim_{R\to \infty} R^{-2} \min_{h \in \mathcal A^p(\alpha_{R\Omega}  \mu_{\mathcal T})} S^p_{R\Omega} (h)  =
\\
\inf\biggl\{\lim_{R\to \infty} R^{-2} \min_{h \in \mathcal A^p(\alpha_{R\Omega}\, \mu)} S^p_{R\Omega}(h) \; : \;
 \mu = \sum_{z \in Z} \delta_z, \;  Z \subset \R^2 \text{ countable }
 \\
\text{and  }
\lim_{R \to \infty} \alpha_{R\Omega}(\mu)=1\biggr\},
\end{multline}
where $\alpha_{\Omega}(\mu) = \frac{|\Omega|}{\mu(\Omega)}$ is a normalization factor.
Fejes T\'oth's result in the case $f(r)=|r|$ in equation \eqref{minpb:location} implies that the conjecture is true if $p=\infty$ and $\Omega$ is a polygonal domain with at most 6 sides.

The definition of $S^\infty_\Omega$ is motivated by the observation that
the minimum of $S^p_\Omega$ over $\mathcal A^p(\mu)$ admits an unconstrained variational characterization if $p>2$.
\begin{proposition}
\label{unconstr}
Let $\mu$ be an atomic measure such that $\mu(\Omega) = |\Omega|$.
Then
\begin{align} \label{equival}
\min_{h \in \mathcal A^p(\mu)} S^p_\Omega(h) = \Gamma_\Omega^p(\mu)
\end{align}
holds for all $2<p\leq \infty$, where
\begin{align*}
& \Gamma^p_\Omega(\mu) =
\\
& \left\{\begin{array}{ll}
\displaystyle \sup\left\{\int_{\Omega} \phi \, \dd \mu - \frac{1}{p} \int_{\Omega} |\nabla \phi|^p \, \dd x  \; : \; \phi \in W^{1,p}(\Omega),\;\int_{\Omega} \phi = 0 \right\} \; &  2<p< \infty,\\[1em]
\displaystyle \sup\left\{\int_{\Omega} \phi\, \dd \mu \; : \; \|\nabla \phi\|_{L^\infty}\leq 1 , \; \int_{\Omega} \phi = 0\right\} \; & p=\infty.
\end{array}\right.
\end{align*}
\end{proposition}

Proposition~\ref{unconstr} provides a homotopic connection between the physically interesting functional $S^2_\Omega$ and the functional $S^\infty_\Omega$ for which mathematically rigorous analysis of the asymptotic behavior of minimizers is available. The presence of the connection suggests that the Abrikosov lattice is optimal for sufficiently large $p$ and offers a strategy for the construction of rigorous mathematical proofs.
The proof is given in Section~\ref{sec:lemmaproofs}.

%
%
%
%
%
%

\section{Preliminaries}
\label{sec:preliminaries}

\subsection{Cells and an alternative formulation}
A central concept in this work is that of \emph{cells}, which can be seen in Figure~\ref{fig:NumericsSmallLambda}.
These cells arise from the definition~\eqref{def:W}  of~$\dw$: the cell associated with any $z\in Z$ is the set $T^{-1}(z)$, where $T$ is the optimal map in~\eqref{def:W}. In Lemma~\ref{lemma:polygons2} we show that for any~$\mu$, these cells are separated by straight lines, and Figure~\ref{fig:NumericsSmallLambda} illustrates this. Note that when the cells are exactly hexagonal, the points $z$ are arranged in a triangular lattice, and vice versa.

In fact there is a useful alternative formulation of this system in terms of the cells themselves. Define the set of \emph{partitions} $\mathcal C(\Omega)$ of $\Omega$  by
\begin{equation}
\label{def:C}
\mathcal C(\Omega) = \left\{{\bm\chi} \in L^\infty(\Omega;\{0,1\})^n \text{ for some }n\geq 1 \; : \;  \sum_{i=1}^n \chi_i = 1 \text{ on }\Omega
 \right\}.
\end{equation}
Now define the alternative energy functional
\[
F_\lambda: C(\Omega_\lambda) \to \R, \qquad
F_\lambda({\bm\chi}) =  \sum_{i=1}^n \left[ 2c_6 \left( \int_{\Omega_\lambda}\chi_i \right)^{1/2} + I_{\Omega_\lambda}(\chi_i)\right],
\]
where $I_U(\chi)$ is defined for any set $U \subseteq \mathbb{R}^2$ and function $\chi \in L^\infty(U;\{0,1\})$ by
\begin{equation}
\label{def:I}
I_U(\chi) = \inf_{\xi \in U} \int_{U} |x-\xi|^2\,\chi(x)\, \dd x.
\end{equation}

The formulations in terms of $\mu$ and of $\bm\chi$ are strongly related. One can construct one out of the other as follows:
\begin{itemize}
\item Given $\mu$  with support $Z$ and transport map $T$, define the partition $\bm \chi$ by setting, for each $z\in Z$, $\chi_z$ to be the characteristic function of the set $T^{-1}(z)$, so that $\int \chi_z = \mu({z})$;
\item Given $\bm\chi=\{\chi_i\}_{i\in I}$, let $z_i$ achieve the infimum in~\eqref{def:I}; then define $\mu = \sum_{i\in I} \big(\int\chi_i\big)\delta_{z_i}$.
\end{itemize}
There is loss of information going from one to the other, and in general this transformation does not preserve energy. However, minimizers are mapped to minimizers, as the following calculation shows. Given a $\mu$, construct the corresponding $\bm\chi$ as above; then
\begin{align}
E_\lambda(\mu) &= 2c_6 \sum_{z\in Z} \mu(\{z\})^{1/2} + \int_{\Omega_\lambda} |x-T(x)|^2 \, \dd x\notag\\
&= 2c_6 \sum_{z\in Z} \Bigl(\int_{\Omega_\lambda} \chi_z\Bigr)^{1/2}
   + \sum_{z\in Z} \int_{\Omega_\lambda} |x-z|^2\, \chi_z(x) \, \dd x\notag\\
&\geq 2c_6 \sum_{z\in Z} \Bigl(\int_{\Omega_\lambda} \chi_z\Bigr)^{1/2}
   + \sum_{z\in Z} \inf_{\xi \in \Omega_\lambda} \int_{\Omega_\lambda} |x-\xi|^2\, \chi_z(x) \, \dd x
   \;=\; F_\lambda(\bm \chi).
\label{ineq:EgeqF}
\end{align}
The inequality above becomes an identity if we minimize the left-hand side over all choices of the support points $Z$ of $\mu$. It follows that $\inf_\mu E_\lambda = \inf_{{\bm \chi}} F_\lambda$, and that minimizers are converted into minimizers.

\subsection{Cells can be assumed to be polygonal}

The following lemma in optimal transportation theory shows how the minimization in the definition~\eqref{def:W} of $\dw$ causes cells to be polygonal.

\begin{lemma}[Cells are polygonal]
\label{lemma:polygons2}
Let $\mu\in \M_{\geq0}(\Omega_\lambda)$ be atomic with $\mu(\Omega_\lambda) = |\Omega_\lambda|$ and define $Z = \mathrm{supp}(\mu)$. Let $T$ be the optimal transport map for
$\dw(\LL_{\Omega_\lambda}, \mu)$.
Then there exists numbers $\ell_z \in \mathbb{R}$ such that for all $z \in Z$
\begin{equation}
\label{char:Tz2}
T^{-1}(z) = \{ x \in \Omega_\lambda : \ell_z + |x-z|^2 \le  \ell_{z'} + |x-z'|^2 \textrm{ for all } z' \in Z \}.
\end{equation}
Moreover, if $\mu$ minimizes $E_\lambda$, then $\ell_z = c_6 \mu(\{z\})^{-1/2}$ (up to a constant that is independent of $z$; note that the right-hand side of \eqref{char:Tz2} is invariant under the addition of the same constant to $\ell_z$ and $\ell_{z'}$).
\end{lemma}

The characterization~\eqref{char:Tz2} implies that for given $\mu$, the corresponding cells can be characterized as the intersection of $\Omega_\lambda$ with a finite number of half-planes. Cells that do not meet the boundary $\partial\Omega_\lambda$ are therefore convex polygons; cells adjacent to a piece of curved boundary have a mixture of straight and curved sides. In this paper we refer to both cases as \emph{convex polygons}.

A characterization related to~\eqref{char:Tz2} appears in a number of places \cite{AurenhammerHoffmannAronov98,Merigot11}, and can be proved using Brenier's theorem characterizing optimal transport~\cite{Brenier91}. It shows that the transport cells
$T^{-1}(z)$ form the power diagram of the set of points $Z$ with weights $-\ell_z$, and provides a link between optimal transportation theory and computational geometry. Since Lemma~\ref{lemma:polygons2} is a slightly stronger statement, we give an independent proof in Section~\ref{sec:lemmaproofs}.

\subsection{Optimal energy for polygons}
We first discuss the minimum energy for polygonal domains. Define the number
\begin{align}
c_n
& = \inf_\chi \big\{I_{\mathbb{R}^2}(\chi): \chi \text{ is the characteristic function of an $n$-gon with area } 1 \big\} \nonumber \\
\label{cwdef}
& = \inf_{P} \left\{ \min_{\xi \in P} \int_{P} |x-\xi|^2 \, \dd x : P \textrm{ is an $n$-gon with area } 1 \right\}.
\end{align}
A classical result by L.~Fejes T\'oth \cite[p.~198]{Fejes-Toth72} states that the minimizing $n$-gon is a
regular $n$-gon:
\begin{lemma}[Regular polygons are optimal] \label{regpoly}
The minimum in (\ref{cwdef}) is attained by a regular polygon with $n$ sides, and in particular
\begin{align} \label{cformula}
c_n &= \frac{1}{2n}\left(\frac{1}{3}\tan\frac{\pi}{n} +
\cot\frac{\pi}{n}\right).
\end{align}
The minimum is unique up to rotation
and translation.
\end{lemma}
Note that the number $c_6$, defined in \eqref{def:c6}, equals $c_n$ for $n=6$.
If $\chi$ is the characteristic function of a regular $n$-gon with volume $v$ contained in a domain $\Omega$, then
$I_{\Omega}(\chi) = v^2 c_n$. By Lemma \ref{regpoly}, if $\chi$ is the characteristic function of an \emph{irregular} $n$-gon with volume $v$ contained in a domain $\Omega$, then
\begin{equation}
\label{eq:I}
I_{\Omega}(\chi) \ge v^2 c_n.
\end{equation}

G. Fejes T\'oth proved a stability result for a large number of polygons that applies to the situation at hand. We reproduce a consequence of the main theorem of~\cite{G-Fejes-Toth01} here:
\begin{lemma}[Geometric stability]
\label{lem:GFT}
Let $\Omega$ be a polygon of unit area with at most six sides, and let $\lambda>0$. Let $\bm\chi = \{\chi_i\}_{i=1,\dots,N}$ be a polygonal partition of $\Omega_\lambda$. Set
\[
\e := \frac{\sum_i I_{\Omega_\lambda}(\chi_i) - Nc_6}{Nc_6}.
\]
There exists $\e_0$ and $C>0$ such that if $0<\e<\e_0$ then the following holds. Except for at most $C\e^{1/3}$ indices $i$, all $\chi_i$ are $O(\e^{1/3})$ close to unit-area regular hexagons, in the sense that $\supp\chi_i$ is a hexagon and the distances from the center of mass to the vertices and to the sides are between $(1-C\e^{1/3})$ and $(1+C\e^{1/3})$ of the corresponding values for a unit-area regular hexagon.
\end{lemma}

Note that this lemma implies a similar statement on the centers of mass: if $z_i$ is the center of mass of $\chi_i$, thus achieving the minimum in~\eqref{def:I}, then apart from a fraction $C\e^{1/3}$, all of the $z_i$ have exactly six neighbours at distance $(1\pm C\e^{1/3})$ of the optimal lattice spacing.

\subsection{The average number of edges of a polygonal cell}
Lemma \ref{lemma:polygons2} shows that the optimal transport map $T$ gives rise to a partition of $\Omega_\lambda$ by convex polygons.
Therefore Euler's polytope formula applies:
\[
\textsl{vertices} - \textsl{edges} + \textsl{faces} = 2.
\]
In the proofs of Theorems \ref{th:lowerbound} and \ref{th:stability} we will use the following lemma, which follows from Euler's polytope formula:
\begin{lemma}[Bound on the average number of edges of the polygons] \label{lem:combinatorics}
Assume that $\Omega\subset \R^2$ is a polygon with at most six sides. Consider a partition of $\Omega$ by convex polygons. Then the average number of edges per polygon is less than or equal to six.
\end{lemma}
\begin{proof}
 A proof of this is given in Morgan \& Bolton (2002, Lemma 3.3) for the case where $\Omega$ is a square. The proof for $3$-, $5$- and $6$-gons is almost identical, and we only give it here for completeness.

Let $S \in \{3,4,5,6\}$ denote the number of sides of $\Omega$.
Let the tiling of $\Omega$ consist of $n$ convex polygons $P$. Denote the number of edges of polygon $P$ by $N(P)$. Let $N_0$ denote the number of exterior edges.

All the interior edges meet two faces, whereas the exterior edges meet only one face. Therefore the total number of edges $e$ can be written as
\begin{equation}
e = \sum_P \frac{N(P)}{2} + \frac{N_0}{2}.
\end{equation}
Since the tiles are convex, each interior vertex lies on at least three faces. The exterior vertices, except possibly the $S$ corners of $\Omega$, lie on at least two faces.
Therefore we can bound the total number of vertices $v$ by
\begin{equation}
\label{eq:vb}
v \le \sum_P \frac{N(P)}{3} + \frac{N_0}{3} + \frac{S}{3}.
\end{equation}
Euler's formula gives
\begin{equation}
2 = v-e+(n+1) \le \left(\sum_P \frac{N(P)}{3} + \frac{N_0}{3} + \frac{S}{3} \right) - \left( \sum_P \frac{N(P)}{2} + \frac{N_0}{2} \right)
+ (n+1).
\end{equation}
Since $N_0 \ge S$ it follows that
\begin{equation}
\label{eq:N}
\frac{1}{n} \sum_P N(P) \le 6 - \frac{6-S}{n} \leq 6.
\end{equation}
\end{proof}

For the proof of Theorem \ref{th:limit}, where $\partial \Omega$ is the image of a Lipschitz function, we will need a different version of Lemma \ref{lem:combinatorics}:
\begin{lemma}[Bound on average number of edges for a planar graph] \label{lem:combinatorics2}
Let $G$ be a planar graph such that the degree of each vertex is at least three. Then
the average number of edges per face is less than six.
\end{lemma}
\begin{proof}
This is proved by simply taking $S=0$ in \eqref{eq:vb}--\eqref{eq:N} in the proof of Lemma \ref{lem:combinatorics}.
\end{proof}
\subsection{L.~Fejes T\'oth's Theorem}
\label{subsection:FT}
For pedagogical purposes we consider a simpler setting where the surface energy, the first term of $\hat{E}_\lambda$, is dropped, i.e., the case
$\lambda = 0$. Roughly speaking, if the number of points in $Z$ is fixed beforehand, then minimizers of $\hat{E}_0$ tend to a triangular lattice as $|Z| \to \infty$. This is essentially a special case of a classic result by L.~Fejes T\'oth \cite{Fejes-Toth72}, which we give a short proof of here.
\begin{theorem} \label{thm:wasslatt}
Let $\Omega \subset \R^2$ be a polygon with at most 6 sides such that $|\Omega|=1$. Then
$$
\#\mathrm{supp}(\mu)\,W(\LL_{\Omega},\mu) \geq c_6
$$
for all atomic probability measures $\mu$ that are supported on a finite set.
Moreover
\begin{align} \label{eq:ub}
\inf\{ \#\mathrm{supp}(\mu)\,W(\LL_{\Omega},\mu)  :
 \mu \in \mathcal{P}(\Omega) \text{ is atomic with finite support}\} = c_6.
\end{align}
\end{theorem}
\begin{proof}
Let $\mu$ be an atomic measure.
By Lemma~\ref{lemma:polygons2} the  characteristic functions $\chi_z$ are supported on polygonal domains, $T^{-1}(z)$.
Let $n_z \in \{3,4,\ldots \}$ be the number of sides of $T^{-1}(z)$.
Lemma~\ref{regpoly} implies that we can reduce the energy of $\mu$ by replacing each
polygon $T^{-1}(z)$ with a \emph{regular} polygon with the same number of sides and the same area:
\begin{align*}
W(\LL_\Omega,\mu) = \sum_{z \in Z} \int_{\Omega} |x-z|^2 \chi_z \, \dd x \ge \sum_{z \in Z} I_\Omega(\chi_z) \geq \sum_{z \in Z} v_z^{2} c_{n_z}
\end{align*}
by equation \eqref{eq:I}, where $v_z= \int\chi_z$.
Define $\kappa=\frac{\partial c_n}{\partial n}|_{n=6}=\frac{2\pi}{243} - \frac{5 \sqrt{3}}{324}<0$. Define
$ g(v,n) = v^2 c_n$. By computing the Hessian of $g$ one can show that $g$ is convex in $(v,n)$:
\[
\textrm{det} (D^2 g) = \frac{8 \pi^2 v^2 \sec^2 \left( \tfrac{\pi}{n}\right)}{9 n^6} > 0, \quad \frac{\partial^2 g}{\partial v^2} = 2 c_n > 0.
\]
Hence
for each $v_0\geq 0$ one finds that
\begin{equation}
\label{eq:cnvx}
v^2 c_n \geq c_6 v_0^2 + 2 v_0 c_6 (v-v_0) + \kappa v_0^2 (n-6)
\end{equation}
for all $v\geq 0$, $n \in \{3,4,\ldots\}$.
This implies that
\begin{align}
\nonumber
W(\LL_\Omega, \mu) \geq& \sum_{z \in Z} \left(c_6 v_0^2 + 2 v_0 c_6 (v_z-v_0)
+\kappa v_0^2(n_z-6)\right)\\
\label{eq:lb}
= &  c_6 v_0^2 |Z| +2 v_0 \,c_6 -2 v_0^2 c_6 |Z| +
\kappa \,v_0^2\left(\sum_{z}n_z - 6 |Z| \right),
\end{align}
where we have used that $\sum v_z = |\Omega| = 1$. Substituting $v_0 = |Z|^{-1}$ into \eqref{eq:lb} gives
\[
W(\LL_\Omega, \mu) \geq  \frac{c_6}{|Z|} +
\kappa \,v_0^2\left(\sum_{z}n_z - 6 |Z| \right).
\]
Lemma~\ref{lem:combinatorics} implies that $\sum_{z}n_z \leq 6 |Z|$. Recall also that $\kappa < 0$. Therefore we conclude that $W(\LL_\Omega,\mu) \geq \frac{c_6}{|Z|}$ as required.

To prove the upper bound we define $\chi_z$ to be the characteristic functions of the Voronoi-tessellation
of $\R^2$ that is associated with the set $Z_m = m^{-\frac{1}{2}} \TT \subset
\R^2$, $m \in \mathbb{N}$, where $\mathcal{T}$ is the triangular lattice defined in equation \eqref{def:T}.
We will check that the following sequence of probability measures $(\mu_m)_{m=1}^\infty$ achieves the infimum in \eqref{eq:ub}:
$$ \mu_m = \sum_{z \in Z_m} \delta_z \int_\Omega\chi_z.$$
It is easy to check that $v_z:=\int_\Omega \chi_z =m^{-1}$ if $\supp \chi_z \subset
\Omega$ and $v_z < m^{-1}$ otherwise.
Also, for all $z \in Z_m$, we have
\begin{equation}
\label{eq}
\int_\Omega |x-z|^2 \chi_z \, \dd x \le \frac{c_6}{m^2},
\end{equation}
with equality if $\mathrm{supp} \chi_z \subset \Omega$.
Furthermore it can be shown that
\begin{align} \label{eq:bdrycells}
b(m)=\#\left\{z \in Z \; : \; \emptyset \neq \mathrm{supp} \chi_z\cap \Omega
\neq \mathrm{supp}\chi_z \right\}\leq C m^{\frac{1}{2}}
\end{align}
for some universal constant $C$ (which depends on $\mathcal{H}^1(\partial \Omega$)).
Therefore by \eqref{eq} we obtain
\begin{align*}
 W(\LL_\Omega,\mu_m) \leq
 \sum_{z \in Z_m} \int_\Omega |x-z|^2 \chi_z \, \dd x \le
 (m+b(m)) \frac{c_6}{m^2} \le \frac{c_6}{m} + C \frac{c_6}{m^{3/2}}.
\end{align*}
Since $\# \mathrm{supp}(\mu_m) \leq m + b(m)$ this proves that the lower bound (\ref{eq:ub}) can be achieved with the sequence $\mu_m$.
\end{proof}

\paragraph{Remark} We will see that the proof of Theorem \ref{th:lowerbound} mimics the proof of Theorem \ref{thm:wasslatt}. The important difference is that for $E_\lambda$ the function $f(v,n)$ that corresponds to $g(v,n)$ in the proof of Theorem \ref{thm:wasslatt} is \emph{not convex} (see equation \eqref{def:f} for the definition of $f$). We circumvent this lack of convexity by proving that a convexity inequality of the form \eqref{eq:cnvx} still holds if $v$ is sufficiently large: $v \ge m_1$ (Lemma \ref{lem:conv}). Then we prove in Lemma \ref{lem:minv} that if $\mu$ is a minimizer of $E_\lambda$, then $v_z > m_1$ for all $z$ and so the convexity inequality applies.

%
%
%
%

\section{Proofs of Theorems~\ref{th:lowerbound}, \ref{th:stability}, and~\ref{th:periodic}}

In this section we give the proofs of Theorems~\ref{th:lowerbound}, \ref{th:stability}, and~\ref{th:periodic}, postponing certain results to later lemmas when necessary. As in the hypotheses of Theorems~\ref{th:lowerbound} and~\ref{th:stability}, we first assume that $\Omega$ is a polygon with at most six sides. Note that therefore all cells are also polygons (by Lemma \ref{lemma:polygons2}).

Throughout this section, let $\mu$ be a minimizer of $E_\lambda$, $(\chi_z)_{z\in Z}$ be the partition generated by $\mu$, $v_z = \mu(\{z\})$ for $z\in Z$, and $n_z$ be the number of sides of $\supp \chi_z$.

\medskip
The proofs of Theorems~\ref{th:lowerbound} and~\ref{th:stability} make use of the following ingredients:
\begin{enumerate}
\item A pseudo-localization result implied by inequalities~\eqref{ineq:EgeqF} and \eqref{eq:I} that decouples the atoms and cells from each other:
\begin{equation}
\label{eq:Elb}
E_\lambda (\mu)\geq 2c_6 \sum_{z\in Z} v_z^{1/2}
   + \sum_{z\in Z} I_{\Omega_\lambda}(\chi_z)
\geq \sum_{z\in Z} f(v_z,n_z),
\end{equation}
where
\begin{align}
\label{def:f}
f(v,n) & := 2\c6v^{1/2}+\c nv^2.
\end{align}
\item \label{list:lowerboundf}
A lower bound on the function $f$ that is sharp at $v=1$ and $n=6$.
\item Euler's polytope formula (Lemma~\ref{lem:combinatorics}), which limits the average of $n_z$ to six.
\end{enumerate}

Taking into account these properties, Theorem~\ref{th:lowerbound}  reduces to the statement
\begin{equation}
\label{ineq:lowerbound2}
\inf \left\{\sum_{z\in Z} f(v_z,n_z) : Z\text{ finite}, \ \sum_{z\in Z} v_z = V_\lambda, \ \sum_{z\in Z} n_z \leq 6|Z| \right\} \geq 3\c6 V_\lambda,
\end{equation}
and the first part of Theorem~\ref{th:stability} to the statement that equality in this lower bound implies that $v_z=1$ and $n_z=6$ for all $z$. Without the $n$-dependence the inequality~\eqref{ineq:lowerbound2} and the characterization of minimizers have been proved in~\cite{ChoksiPeletier10}; Lemma~6.2 in this reference shows that
\[
\inf \left\{\sum_{z\in Z} \big[v_z^{1/2} + v_z^2\big] : Z\text{ finite}, \ \sum_{z\in Z} v_z \text{ given}\right\}
\]
is only achieved for constant $v_z$. The proofs below extend this statement to include the $n$-dependence.

\medskip

Ingredient~\eqref{list:lowerboundf} above is the following:
\begin{lemma}[Lower bound on $f$] \label{lem:conv}
There exist $\xi,\zeta>0$ such that the function
$f$  in~\eqref{def:f}
satisfies the bound
\begin{align} \label{convexity}
f(v,n)-3\c6v + \kappa(6-n)\geq \xi (v-1)^2 + \zeta\left(\frac{1}{n}-\frac{1}{6}\right)^2
\end{align}
for all $n\in \{3,4,\ldots\}$, $v \geq m_1 :=1.5\cdot10^{-4}$, where $\kappa := \partial_n c_n\bigr|_{n=6} = \frac{2\pi}{243} - \frac{5 \sqrt{3}}{324}$.
\end{lemma}

While $f(v,n)$ is not convex, Lemma \ref{lem:conv} says that the convexity inequality \eqref{convexity} holds if $v$ is large enough, $v \geq m_1$.
The following lemma shows that for minimizers $\mu$ we do indeed have $v_z \geq m_1$ for all $z \in Z$, which will allow us to apply Lemma
\ref{lem:conv} to prove Theorem \ref{th:lowerbound} following the same strategy as the proof of Theorem \ref{thm:wasslatt}.

\begin{lemma}[Bounds on holes and masses] \label{lem:minv}
Let $\mu$ be a minimizer of $E_\lambda$.
\begin{itemize}
\item[(i)] For all $z\in Z$,  $v_z\geq m_0:=2.4095 \cdot 10^{-4}$.
\item[(ii)] Let $z_0 \in Z$. If the ball $B_R(z_0)$ satisfies $B_R(z_0) \cap Z = \{ z_0 \}$, then $R < R_0$, where
 $R_0= 3.2143$.
 \item[(iii)] Let $B$ be a ball of radius $R$. If it satisfies $B \cap Z = \emptyset$, then $R < R_0$.
 \item[(iv)]
 For all $z \in Z$,
 $\mathrm{diam}(T^{-1}(z)) < D_0$,
  where $D_0^2 = 4 [c_6 m_0^{-1/2} + R_0^2]$.
 \end{itemize}
\end{lemma}

\paragraph{Remark} Note that all the constants $m_0, R_0, D_0$ are independent of $\lambda$.
The constant $R_0$ in part (iii) can be easily improved (see the proof of Lemma \ref{lem:minv}), but this is not necessary for our purposes.

We can now wrap up the proofs of Theorem~\ref{th:lowerbound} and~\ref{th:stability}.

\begin{proof}[Theorem~\ref{th:lowerbound}]
By the arguments above, we only need to prove~\eqref{ineq:lowerbound2}.
Lemma~\ref{lem:minv} implies that $v_z \geq m_0> m_1$, and  Lemma~\ref{lem:conv} gives the inequality
\begin{equation}
\label{LB0}
\begin{aligned}
V_\lambda ^{-1}\sum_{z\in Z} f(v_z,n_z)  \geq V_\lambda^{-1} \sum_{z \in Z} \left(3\c6v_z - \kappa(6-n_z)\right)
 = 3\c6- \kappa V_\lambda^{-1} \sum_{z \in Z}( 6-n_z).
\end{aligned}
\end{equation}
Lemma \ref{lem:combinatorics} and the fact that $\kappa < 0$ imply that the second term on the right-hand side of \eqref{LB0} is non-negative. In this way we arrive at the desired lower bound
\[
V_\lambda^{-1} E_\lambda(\mu) \geq 3\c6.
\]
This concludes the proof of Theorem~\ref{th:lowerbound}.
\end{proof}

\begin{proof}[Theorem~\ref{th:stability}]
By Lemmas~\ref{lem:conv}, \ref{lem:minv}, and~\ref{lem:combinatorics} we find, as in Theorem~\ref{th:lowerbound}, that
\begin{align}
\nonumber
\frac1{V_\lambda} \sum_{z\in Z} \left[ \xi(v_z-1)^2 + \zeta\left(\frac1{n_z}-\frac16\right)^2\right]
&\leq
\frac1{V_\lambda} \sum_{z\in Z} \big[ f(v_z,n_z) + \kappa (6-n_z)\big] - 3\c6\\
\label{eq:d}
&\leq V_\lambda^{-1} E_\lambda(\mu) - 3\c6 = d(\mu).
\end{align}
In the first assertion of the theorem the right-hand side is zero, and therefore $v_z=1$ and $n_z=6$ for all $z\in Z$. Since each cell achieves the minimum in~\eqref{cwdef}, by Lemma~\ref{regpoly} each cell is a regular hexagon of area $1$. This proves the first part of Theorem~\ref{th:stability}.

To prove the second part we will apply Lemma~\ref{lem:GFT}, which requires an estimate of
\[
\e := (N\c6)^{-1}\sum_{z} I_{\Omega_\lambda}(\chi_z) - 1
\]
in terms of the defect $d(\mu)$. Here $N=|Z|$. We first prove some auxiliary estimates.

We calculate that
\[
\frac {V_\lambda}N\left(\frac{{V_\lambda}-N}{V_\lambda}\right)^2 = \frac {V_\lambda}{N}\left(\frac1{V_\lambda} \sum_{z\in Z} (v_z-1)\right)^2
\leq  \frac 1{V_\lambda} \sum_{z\in Z} (v_z-1)^2 \leq \frac1\xi d(\mu)
\]
by equation \eqref{eq:d}.
Since the left-hand side equals ${V_\lambda}/N - 2 + N/{V_\lambda} \ge V_\lambda/N - 2$, this implies that ${V_\lambda}/N \leq 3$ if $d(\mu)$ is small enough.
Also, since $\sqrt x \geq  \frac12 + \frac12 x  -\frac12 (x-1)^2$,
\[
\sum_{z\in Z} v_z^{1/2} \geq \frac12 N+ \frac12 \sum_{z\in Z} v_z - \frac12 \sum_{z\in Z} (v_z-1)^2
\geq \frac12 N + \frac12 {V_\lambda} - \frac {V_\lambda}{2\xi} d(\mu).
\]
Finally,
\[
\frac1N \sum_{z\in Z} (v_z-1) \leq \bigg(\frac1N \sum_{z\in Z} (v_z-1)^2  \bigg)^{1/2} \leq \Big(\frac {V_\lambda}N\Big)^{1/2} \Big(\frac1\xi d(\mu)\Big)^{1/2}.
\]
Combining all these inequalities and using equation \eqref{eq:Elb} we estimate
\begin{align*}
\e = \frac1{N\c6} \sum_{z\in Z} I_{\Omega_\lambda}(\chi_z) - 1
& \le \frac1{N\c6} (E_\lambda(\mu) - 3\c6{V_\lambda}) + \frac3N {V_\lambda}- 1- \frac2N \sum_{z\in Z} v_z^{1/2}\\
&\leq \frac1{N\c6} {V_\lambda}d(\mu) +  \frac3N {V_\lambda}- 1 -1 - \frac {V_\lambda}N + \frac {V_\lambda}{\xi N} d(\mu)\\
&= \frac {V_\lambda}N \Big(\frac 1{\c6}  + \frac 1{\xi }\Big) d(\mu) + \frac2N \sum_{z\in Z} (v_z-1) \\
&\leq \frac {V_\lambda}N \Big(\frac 1{\c6}  + \frac 1{\xi }\Big) d(\mu) + 2\Big(\frac {V_\lambda}N\Big)^{1/2} \Big(\frac1\xi d(\mu)\Big)^{1/2}.
\end{align*}
For small enough $d(\mu)$, ${V_\lambda}/N \leq 3$ as mentioned above, and the inequality above reduces to
\[
\e \leq C d(\mu)^{1/2}
\]
for some constant $C$. An application of Lemma~\ref{lem:GFT} then concludes the proof.
\end{proof}

\medskip

The proof of Theorem~\ref{th:periodic} follows along very similar lines to that of Theorem~\ref{th:stability}(a). The energy $\Eper_\lambda(\mu)$ is again bounded from below by the energy of the cells (inequality~\eqref{ineq:EgeqF}), once one replaces the Euclidean distance $|\cdot|$ by the periodized metric $d(\cdot,\cdot)$. The fact that regular $n$-gons are optimal among all $n$-gons (inequality~\eqref{eq:I}) holds similarly, since the requirement that a polygon `fits in the periodic domain' only implies an additional restriction on the polygon, that is not represented in $c_n$. Therefore the inequality~(\ref{eq:Elb}-\ref{def:f}) again applies, and by the same argument as in the proof of Theorem~\ref{th:stability} (where now $d(\mu)=0$) it follows that  $v_z=1$ and $n_z=6$ for all $z\in Z$. This proves the theorem.

%
%
%
%

\section{Proof of Theorem~\ref{th:limit}}

The following lemma is proved (see Section~\ref{sec:lemmaproofs}) by constructing a trial function:
\begin{lemma}[Upper bound on the minimal energy]
\label{lem:hextrial}
\label{UB}
Let $\partial \Omega=\varphi(K)$ for some Lipschitz function $\varphi:\mathbb{R}\to \mathbb{R}^2$ and compact set $K \subset \mathbb{R}$.
Then there exists $\lambda_0>0$ such that for all $0<\lambda<\lambda_0$
\begin{equation}
\nonumber
\inf_{\bm \chi} F_\lambda (\bm\chi) \le 3 \, \c6 V_\lambda + C \, \mathcal{H}^1 (\partial \Omega_\lambda),
\end{equation}
where  $C = 2^\frac{5}{2} 3^{\frac{1}{4}} \c6 (1+\eta),$ and $\eta>0$ can be made arbitrarily small by taking $\lambda_0$ small enough.
\end{lemma}
\noindent
This result proves the upper-bound part of Theorem~\ref{th:limit}, since
\begin{align}
\notag
V_\lambda^{-1} \inf E_\lambda = V_\lambda^{-1} \inf F_\lambda
&\leq 3 \c 6 + C V_\lambda^{-1}\,\mathcal H^1(\partial\Omega_\lambda)\\
&= 3 \c 6 + C V_\lambda^{-1/2}\,\mathcal H^1(\partial\Omega) \longrightarrow 3\c6 \textrm{ as } \lambda \longrightarrow 0.
\label{ineq:upper}
\end{align}

\medskip

The specific characterization of the boundary $\partial\Omega$ in terms of a Lipschitz mapping stems from the following useful result. If $J_r(A) := A+ B(0,r)$ is the tube of radius $r$ around the set $A$, then this characterization of $\partial \Omega$ implies that
\begin{equation}
\label{char:Mink}
\lim_{r\to0} \frac1{2r} |J_r(\partial\Omega)| = \mathcal H^1(\partial \Omega)
\end{equation}
(see~\cite[Th. 2.106]{AmbrosioFuscoPallara00}). We use this below and in the proof of Lemma~\ref{lem:hextrial}.

\medskip
To conclude the proof of Theorem~\ref{th:limit} we derive a matching lower bound. Note that we cannot use Theorem~\ref{th:lowerbound} for this, since in Theorem~\ref{th:limit} the domain $\Omega$ need not be a polygon.

Take a minimizer $\mu$ of $E_\lambda$ and let
 $(\chi_z)_{z \in Z}$ be the corresponding partition.
By Lemma \ref{lem:minv}, (iv),
\begin{align} \label{updbd}
\mathrm{diam}(\supp(\chi_z)) < D_0 \text{ for all } z \in Z.
\end{align}
Let $\partial Z\subset Z$ be the set of those points $z$
such that $\partial \, \supp(\chi_z) \cap \partial \Omega_\lambda \neq \emptyset$.
The bound~(\ref{updbd}) implies that
\begin{equation}
\nonumber
\mathrm{dist} \left( \bigcup_{z \in Z \setminus \partial Z} \supp(\chi_z),\partial \Omega_\lambda \right) < D_0.
\end{equation}
Therefore, using~\eqref{char:Mink}, it follows that there exists a constant $\lambda_0>0$ such that for all
$\lambda < \lambda_0$,
\begin{align} \label{sumbd}
\sum_{z \in Z\setminus \partial Z} v_{z}
\geq V_\lambda - |J_{D_0} (\partial \Omega_\lambda ) |
\geq V_\lambda - C \mathcal{H}^1(\partial \Omega_\lambda)
\end{align}
for some constant $C>0$ that is independent of $\lambda$.
Note that equation~(\ref{sumbd}), the fact that $\sum_{z \in Z} v_z = V_\lambda$, the lower bound $v_z \geq m_0$,
and the fact that $\lim_{\lambda \to 0}\mathcal{H}^1(\partial \Omega_\lambda) V_\lambda^{-1} =0$
 imply that
\begin{equation}
\label{Z}
 \frac{|Z|-|Z \setminus \partial Z|}{V_\lambda} = \frac{|\partial Z|}{V_\lambda} \to 0 \textrm{ as } \lambda \to 0.
\end{equation}
Lemma~\ref{lemma:polygons2} implies that for each $z \in Z \setminus \partial Z$ the support of $\chi_z$ is the interior of a convex polygon. Let $n_z$ be the number of edges of $\partial\supp(\chi_z)$ and $v_z= \mu(\{z\})$.
By combining~\eqref{ineq:EgeqF}
and \eqref{eq:I} we find that
\begin{align*}
E_\lambda(\mu) \geq  \sum_{z \in Z\setminus\partial Z}\left( 2 c_6 \, (v_z)^\frac{1}{2}+  c_n \,v_z^2 \right)=  \sum_{z \in Z\setminus \partial Z} f(n_z,v_z).
\end{align*}
As in the proof of Theorem \ref{th:lowerbound}, since $m_0 > m_1$, Lemma~\ref{lem:conv} implies that
\begin{equation}
\label{LB1}
\begin{aligned}
V_\lambda^{-1}\,E_\lambda(\mu) & \;\geq\;   V_\lambda^{-1}\sum_{z \in Z\setminus \partial Z} \left(3\c6v_z - \kappa(6-n_z)\right)\\
 & \;\geq\; 3\c6\left(1- C\, \mathcal{H}^1(\partial \Omega_\lambda) V_\lambda^{-1}\right)- \kappa V_\lambda^{-1} \sum_{z \in Z\setminus \partial Z}( 6-n_z),
\end{aligned}
\end{equation}
where the second inequality follows from  \eqref{sumbd}.

We now define a planar graph $G$ as follows. Include all edges and vertices of the convex polygons $\mathrm{supp}(\chi_z)$ for $z \in Z \setminus \partial Z$. Now for each $z \in \partial Z$, add nodes and edges to the graph as follows. The partition $\chi_z$ has one or more straight edges that intersect $\partial \Omega_\lambda$. Add these edges to the graph and add the intersection points as nodes. Finally, replace each section of $\partial\Omega_\lambda$ between two such nodes by a single edge. In this way we obtain a planar graph $G$ with one face for each $z \in Z$ such that the degree of each vertex is at least $3$. Let $n_z$ denote the number of edges of face $z$. (This notation is consistent with that given above for $z \in Z \setminus \partial Z$).

Using this construction, the second term on the right-hand side of equation~(\ref{LB1}) satisfies
\begin{equation}
\label{LB2}
\begin{aligned}
- \kappa V_\lambda^{-1} \sum_{z \in Z\setminus \partial Z}( 6-n_z) & \geq - \kappa V_\lambda^{-1} \left( 6 |Z \setminus \partial Z| -  \sum_{z \in Z} n_z \right) \\
& \geq - \kappa V_\lambda^{-1} \left( 6 |Z \setminus \partial Z| -  6 |Z| \right) \, \to 0 \textrm{ as } \lambda \to 0,
\end{aligned}
\end{equation}
where in the second line we have used Lemma \ref{lem:combinatorics2}, equation \eqref{Z}, and the fact that $\kappa<0$.
By combining \eqref{LB1}, \eqref{LB2}
and the fact that $\lim_{\lambda \to 0}\mathcal{H}^1(\partial \Omega_\lambda) V_\lambda^{-1} =0$ we obtain
$$ \liminf_{\lambda \to 0} V_\lambda^{-1}\, \inf E_\lambda \geq 3\c6.$$
Together with (\ref{ineq:upper}) this implies~(\ref{ceq2}) and concludes the proof of Theorem~\ref{th:limit}.

%
%
%
%

\section{Proofs of the Lemmas}
\label{sec:lemmaproofs}

\begin{proof}[Lemma \ref{lemma:polygons2}, that cells are polygonal]
For given $\mu$, let $Z$ be the support set and $T$ the optimal map in \eqref{def:W}. Take an  ordering $z_i$ of $Z$, $i = 1, \ldots , n$, such that $T^{-1}(z_i)$ is adjacent to $\bigcup_{j=1}^{i-1} T^{-1}(z_j)$.
We now construct the $\ell_i$ iteratively. First note that,
by choosing $i$ and $j<i$ such that $T^{-1}(z_i)$ and $T^{-1}(z_j)$ are adjacent,
\begin{equation}
\begin{aligned}
\nonumber
& \essinf_{x \in T^{-1}(z_i)} \ell_j + |x-z_j|^2 - |x-z_i |^2  \\
& \le \ell_j + |x-z_j|^2 - |x-z_i |^2 \quad \forall \; x \in \partial T^{-1}(z_i) \cap \partial T^{-1}(z_j) \\
& \le \esssup_{x \in T^{-1}(z_j)} \ell_j + |x-z_j|^2 - |x-z_i |^2
\end{aligned}
\end{equation}
and therefore
\begin{multline}
\label{iteration}
\inf_{j=1,\ldots ,i-1} \essinf_{x \in T^{-1}(z_i)} \ell_j + |x-z_j|^2 - |x-z_i|^2
\\ \leq
\sup_{j=1,\ldots ,i-1} \esssup_{x \in T^{-1}(z_j)} \ell_j + |x-z_j|^2 - |x-z_i|^2.
\end{multline}
We now start the iteration by setting $\ell_1=0$. We construct $\ell_i$ in terms of $\ell_1 , \ldots , \ell_{i-1}$ one-by-one: if equality is achieved in~\eqref{iteration}, then define $\ell_i$ to be the common value and iterate; otherwise abort the iteration. If the iteration is never aborted, then the characterization \eqref{char:Tz2} is proved because of the following: We have
\begin{equation}
\begin{aligned}
\label{elli}
\ell_i  & = \inf_{j=1,\ldots ,i-1} \essinf_{x \in T^{-1}(z_i)} \ell_j + |x-z_j|^2 - |x-z_i|^2 \\
& = \sup_{j=1,\ldots ,i-1} \esssup_{x \in T^{-1}(z_j)} \ell_j + |x-z_j|^2 - |x-z_i|^2,
\end{aligned}
\end{equation}
and so for all $i$, $x \in T^{-1}(z_i)$ and $j<i$,
\[
\ell_i \le \ell_j + |x-z_j|^2 - |x-z_i|^2
\]
by the first equality in \eqref{elli}. This also holds for all $j>i$ by the second equality in \eqref{elli}.
Therefore $T^{-1}(z_i) \subseteq
\{ x \in \Omega_\lambda : \ell_i + |x-z_i|^2 \le  \ell_{j} + |x-z_j|^2 \textrm{ for all } z_j \in Z \}$.
The opposite inclusion can be shown by contradiction: Suppose there is an $i$ such that $\ell_i + |x-z_i|^2 <  \ell_{j} + |x-z_j|^2$ for all $j$, but $x \not \in T^{-1}(z_i)$. Then $x \in T^{-1}(z_j)$ for some $j$ and so the inclusion we already proved implies the contradiction
$\ell_j + |x-z_j|^2 \le  \ell_{i} + |x-z_i|^2$.

If, on the other hand, the iteration aborts, then by renumbering we can assume (for notational convenience) that it aborts at the first iteration $i=2$. In this case
lack of equality in \eqref{iteration}
implies that
\begin{equation}
\label{iteration2}
\essinf_{x \in T^{-1}(z_2)} |x-z_1|^2 - |x-z_2|^2
<
\esssup_{x \in T^{-1}(z_1)} |x-z_1|^2 - |x-z_2|^2.
\end{equation}
Equation \eqref{iteration2} implies that there exists balls $B_{\epsilon_1}(x_1)$ and $B_{\epsilon_2}(x_2)$ such that
\[
0 < |B_{\epsilon_1}(x_1) \cap T^{-1}(z_1)| = |B_{\epsilon_2}(x_2) \cap T^{-1}(z_2)|,
\]
and $\forall$ $x'_1 \in B_{\epsilon_1}(x_1) $, $x'_2 \in B_{\epsilon_2}(x_2) $,
\begin{equation}
\label{ineq:char}
|x'_1 - z_2|^2 + |x'_2 - z_1|^2 < |x'_1 - z_1|^2 + |x'_2 -z_2|^2.
\end{equation}
Now define
\begin{equation}
\tilde{T}(x) := \left\{
\begin{array}{ll}
z_2 & \textrm{if } x \in B_{\epsilon_1}(x_1) \cap T^{-1}(z_1), \\
z_1 & \textrm{if } x \in B_{\epsilon_2}(x_2) \cap T^{-1}(z_2), \\
T(x) & \textrm{otherwise}.
\end{array}
\right.
\end{equation}
Then $\tilde{T}$ is admissible and \eqref{ineq:char} implies that
\[
\int_{\Omega_\lambda} | x - \tilde{T}(x) |^2 \, \dd x < \int_{\Omega_\lambda} |x - T(x)|^2 \, \dd x,
\]
which contradicts the optimality of $T$.

The explicit value of the Lagrange multiplier $\ell_z$ for minimizers follows from a similar argument in which the masses are not necessarily conserved.
\end{proof}

\bigskip

\begin{proof}[Lemma~\ref{lem:conv}, the lower bound on $f$]
Take $\xi=\zeta=0.001$.
Define
\begin{equation}
g(v,n) = f(v,n)-3 \c6 v + (6-n)\kappa - \xi (v-1)^2 - \zeta\left(\frac{1}{n}-\frac{1}{6}\right)^2.
\end{equation}
We wish to show that $g(v,n) \geq 0$ for all $n\in \{3,4,\ldots\}$, $v \geq m_1$.

First we consider the case $n=6$. Note that
\begin{equation}
\begin{aligned}
g(v,6) & = (v^{\frac12}-1)^2 [(\c6-\xi)v + 2(\c6-\xi)v^{\frac12}-\xi]
\\
& =: (v^{\frac12}-1)^2 p_6(v^{\frac12}),
\end{aligned}
\end{equation}
where $p_6$ is the quadratic polynomial $p_6(u) = (\c6-\xi)u^2 + 2(\c6-\xi)u-\xi$. Let $u_6=0.0031$ be the positive root of $p_6$. This satisfies $u_6^2 < m_1$. Therefore
$g(v,6) \geq 0$ for all $v \geq m_1$.

Now we consider the case $n \geq 8$. Note that $c_n$ is a decreasing function and so $\kappa <0$ and $c_n \ge \lim_{n \to \infty} c_n = \frac{1}{2 \pi}$.
Therefore
\begin{equation}
g(v,n) \geq 2\c6 \,v^\frac{1}{2}+ \frac{1}{2 \pi} v^2 -3\c6 v -2 \kappa - \xi (v-1)^2 - \frac{\zeta}{36} =: p_8(v^\frac12),
\end{equation}
where $p_8$ is the quartic polynomial
\begin{equation}
p_8(u) = \left(\frac{1}{2 \pi} - \xi \right) u^4 + (2 \xi - 3 \c6) u^2 + 2 \c6 u - \left( 2 \kappa + \xi + \frac{\zeta}{36}  \right).
\end{equation}
The discriminant of $p_8$ equals $-2.2\cdot10^{-5}<0$ and so $p_8$ has two real roots and two complex roots. It is easy to check using the Intermediate Value Theorem that both the real roots are negative. Therefore $g(v,n) >0$ for all $n \geq 8$, $v \geq 0$.

The leaves the cases $n=3,4,5,7$, which we check individually.
Define the quartic polynomial
$q_n(u):=g(u^2,n) = au^4 + cu^2 + du + e$
with
\begin{equation}
a = c_n -\xi, \quad c = 2 \xi -3 \c6, \quad d = 2 \c6, \quad e= (6-n) \kappa - \xi  - \zeta\left(\frac{1}{n}-\frac{1}{6}\right)^2.
\end{equation}
The discriminant $\Delta(n)$ of $q_n(u)$ satisfies
\begin{gather}
\Delta(3)=-1.5 \cdot 10^{-3} <0, \quad
\Delta(4)= -2.0 \cdot 10^{-4} <0, \\
\Delta(5) =-2.6 \cdot 10^{-5} <0, \quad
\Delta(7)= -1.3 \cdot 10^{-5} < 0.
\end{gather}
Therefore $q_n(u)$ has two real roots and two complex roots for $n \in \{3,4,5,7\}$. Moreover, since $a>0$ and $e<0$, then $q_n$ has one positive root and one negative root. Using the Intermediate Value Theorem it is easy to check that the positive roots $u_n$ of $q_n(u)$ satisfy
\begin{equation}
u_7 < u_5 < u_4 < u_3 < 0.012.
\end{equation}
Therefore if $v \ge m_1$, then $v^\frac12 \ge m_1^\frac12 > 0.012$ and so $g(v,n)>0$ for $n \in \{3,4,5,7\}$.
\end{proof}

\bigskip

\begin{proof}[Lemma~\ref{lem:minv},  bounds on the size of holes and on the masses]
We start by proving (ii). Let $z_0 \in Z$, $R>0$ and define $B=B_R(z_0)$. We suppose that $B \cap Z = \{z_0\}$.
We first estimate $E_\lambda(\mu)=F_\lambda(\bm \chi)$ from below.
Let $T:\Omega_\lambda \to Z$ be the optimal transportation map for $W(\LL_{\Omega_\lambda}, \mu)$.
Define $\widetilde \chi_{z} = \chi_{z}\, 1_{B^c}$ for all $z \in Z$. Then
\begin{align}
\nonumber
\sum_{z\in Z} I_{\Omega_\lambda}(\chi_{z}) &= \sum_{z\in Z} \int_{\Omega_\lambda} |x-z|^2\, \chi_{z}(x)\,\dd x
= \int_{\Omega_\lambda} |x-T(x)|^2 \, \dd x \\
\nonumber
&= \sum_{z\in Z} \int_{\Omega_\lambda \setminus B} |x-T(x)|^2\,\chi_{z}(x) \, \dd x + \int_B |x-T(x)|^2 \, \dd x\\
\label{ineq}
&\geq\sum_{z\in Z} \int_{\Omega_\lambda} |x-T(x)|^2\,\widetilde\chi_{z}(x) \, \dd x + \int_B \dist(x,\{z_0\}\cup \partial B)^2 \, \dd x\\
\label{eq2}
& \geq \sum_{z\in Z} I_{\Omega_\lambda}(\widetilde \chi_{z}) + \frac{\pi}{12} R^4.
\end{align}
Therefore
\[
F_\lambda(\bm\chi ) =  \sum_{z \in Z} \left\{2\c6( v_{z})^\frac{1}{2} + I_{\Omega_\lambda}(\chi_{z})\right\}\geq \sum_{z \in Z} \left\{2\c6(\widetilde v_{z})^\frac{1}{2} + I_{\Omega_\lambda}(\widetilde \chi_{z})\right\}+ \frac{\pi}{12} R^4.
\]
We now construct a trial partition $\bm{\tilde{\chi}}$ as follows: In $\Omega\setminus B$ the partition is given by $(\widetilde \chi_{z})_{z \in Z}$.
Inside $B$, we take a partition similar to that used in the proof of Lemma~\ref{lem:hextrial}: cover the ball~$B$ with regular hexagons of area $A$ and crop the hexagons at the boundary of $B$ to obtain a partition of~$B$. Let $d_A := 2^{3/2}3^{-3/4} A^{1/2}$ be the diameter of a hexagon of area $A$. The number of hexagons~$N$ needed for the partition satisfies $N \le \pi (R+d_A)^2 / A$. Therefore
\begin{equation}
\begin{aligned}
F_\lambda (\bm{\tilde{\chi}}) & \leq \sum_{z\in Z} \Bigl\{2\c6\, (\widetilde v_{z})^\frac{1}{2} + I_{\Omega_\lambda}(\widetilde v_{z})\Bigr\} + N (2 \c6 A^{1/2} + \c6 A^2),
\\
& \leq \sum_{z\in Z} \Bigl\{2\c6\, (\widetilde v_{z})^\frac{1}{2} + I_{\Omega_\lambda}(\widetilde v_{z})\Bigr\} + \c6 \pi (R+d_A)^2 (2 A^{-1/2} + A).
\end{aligned}
\end{equation}
Since $\mu$ is minimal for $E_\lambda$, $\bm\chi$ is minimal for $F$, and therefore $F(\bm{\tilde \chi})\geq F(\bm\chi)$, which implies that
\begin{equation}
\frac\pi {12} R^4 \leq \c6 \pi (R+d_A)^2 (2 A^{-1/2} + A),
\end{equation}
or
\begin{equation}
R^2 \leq  (R+d_A) [12 \c6 (2 A^{-1/2} + A)]^{1/2} =: q(R;A).
\end{equation}
Let $\hat{R}_0(A)$ be the positive root of the quadratic equation $R^2-q(R;A)=0$. We choose A so that $\hat{R}_0$ is as small as possible. Using computer algebra
\begin{equation}
\min_A \hat{R}_0(A) {}< R_0 := {} 3.2143,
\end{equation}
and the minimum is attained for $A{}\approx{} 0.5820$. Therefore if $B_R(z_0) \cap Z = \{z_0\}$, then $R < R_0 = 3.2143$, as claimed.

The proof of (iii) is the same as the proof of (ii) except that line \eqref{ineq} should be replaced by
\begin{align}
\nonumber
\sum_{z\in Z} I_{\Omega_\lambda}(\chi_{z})
& \geq \sum_{z\in Z} \int_{\Omega_\lambda} |x-T(x)|^2\,\widetilde\chi_{z}(x) \, \dd x + \int_B \dist(x,\partial B)^2 \, \dd x \\
\label{ineq:un}
& \geq \sum_{z\in Z} \int_{\Omega_\lambda} |x-T(x)|^2\,\widetilde\chi_{z}(x) \, \dd x + \int_B \dist(x,\{x_0\}\cup \partial B)^2 \, \dd x,
\end{align}
where $x_0$ is the centre of $B$. The right-hand side equals the right-hand side of equation \eqref{eq2} and the rest of the proof of (iii) is identical to that of (ii). Obviously this proof does not give the sharpest bound on the radius of $B$, due to the unnecessary inequality \eqref{ineq:un}, but it is short and sufficient for our purposes.

We use (ii) to prove (i).
Let $z \in Z$ be such that $v_z$ is minimal. Choose $R=R_0$. Therefore by (ii) there exists $z' \in Z$ with $z'\in B_R(z)\cap \Omega_\lambda$. Define a new partition~$\bm{\widetilde \chi}$ by joining $\chi_z$ and $\chi_{z'}$:
\[
\widetilde \chi_z := \chi_z+\chi_{z'},\quad \widetilde \chi_{z'} := 0, \quad
\widetilde \chi_{z''} := \chi_{z''} \quad \forall \; z'' \in Z \setminus\{z,z'\}.
\]
Upon changing from $\bm\chi$ to $\bm{\widetilde \chi}$, the energy $F$ increases by
\[
a :=   2\c6 (v_z + v_{z'})^\frac{1}{2} + I_{\Omega_\lambda}(\chi_z + \chi_{z'})
- 2\c6\left((v_z)^\frac{1}2 + (v_{z'})^\frac{1}{2}\right) - I_{\Omega_\lambda}(\chi_z) - I_{\Omega_\lambda}(\chi_{z'})
.
\]
Using the concaveness of $x\mapsto \sqrt x$ we estimate that
\[
(v_z+v_{z'})^\frac{1}{2}\leq (v_{z'})^\frac{1}{2} + \tfrac{1}{2} v_z (v_{z'})^{-\frac{1}{2}}.
\]
In the infimum in the definition of $I_{\Omega_\lambda}$, equation \eqref{def:I}, take $\xi = z'$ to obtain
\[
I_{\Omega_\lambda}(\chi_z+\chi_{z'}) \leq I_{\Omega_\lambda}(\chi_{z'}) + \int_{\Omega_\lambda} |x-z'|^2\,\chi_z(x) \, \dd x.
\]
Therefore
\begin{equation}
\label{a}
a \leq 2\c6 \biggl\{ \tfrac1{2} v_z (v_{z'})^{-\frac{1}{2}} - (v_z)^{\frac{1}{2}} \biggl\}+ \int_{\Omega_\lambda} {\chi_z}\bigl[|x-z'|^2 - |x-z|^2\bigr]\, \dd x.
\end{equation}
Note that $z$ is the center of mass of its transport cell $T^{-1}(z)$:
\begin{equation}
\label{com}
zv_z = \int_{\Omega_\lambda} x\chi_z(x)\, \dd x.
\end{equation}
This can be shown by taking the first variation of $E_\lambda$ with respect to $z$.

Expanding the squares in the integral in \eqref{a} and using $v_{z}\leq v_{z'}$ and equation \eqref{com} gives
\[
\frac a{v_z} \leq |z-z'|^2 - \c6 (v_z)^{-\frac{1}{2}}  \leq R^2 - \c6 (v_z)^{-\frac{1}{2}}.
\]
Since $(\chi_z)_z$ is minimal, then $a\geq 0$, and therefore
\[
v_z \geq \frac {\c6^2}{R^4} =  \frac {\c6^2}{R_0^4} \geq{} 2.4095 \cdot 10^{-4}.
\]

Finally we prove (iv). Let $z \in Z$ and $x \in T^{-1}(z)$. By Lemma \ref{lemma:polygons2} and part (i) we obtain
\begin{equation}
\label{D}
|x - z|^2 \le c_6 v_{z'}^{-1/2} - c_6 v_{z}^{-1/2} + |x-z'|^2 < c_6 m_0^{-1/2} + |x-z'|^2
\end{equation}
for all $z' \in Z$. By part (iii), we can find a $\tilde{z} \in Z$ such that $|x-\tilde{z}| \le R_0$. Taking $z '= \tilde{z}$ in equation \eqref{D}
gives
\[
|x - z|^2 < c_6 m_0^{-1/2} + R_0^2.
\]
Therefore $\mathrm{diam}(T^{-1}(z)) < 2 \left( c_6 m_0^{-1/2} + R_0^2 \right)^{1/2} =: D_0$, as required.
\end{proof}

\bigskip

\begin{proof}[Lemma~\ref{lem:hextrial}, the upper bound on the minimal energy]

Let $H$ denote a regular hexagon of area $1$ and let $d:=2\,(3\sin(\frac{\pi}{3}))^{-\frac{1}{2}} = 2^\frac{3}{2}3^{-\frac{3}{4}}$ be its diameter. Let $Z \subset \R^2$ be the centers of a tiling of $\mathbb{R}^2$ by translated copies of $H$, and denote by $H_z \subset \R^2$ the tile centered at $z$.

We construct an upper bound on the minimum energy as follows.
Let $Z(\Omega_\lambda) \subset Z$ be the centers of those hexagons that intersect $\Omega_\lambda$, i.e., $z \in Z(\Omega_\lambda)$ if and only if $H_z \cap \Omega_\lambda \ne \emptyset$. Finally, let $\chi_z$ be the characteristic function of the set $H_z$. Then
\begin{equation}
\label{UB1}
\inf_{\bm\chi} F_\lambda(\chi)
\le F_\lambda \Big( \big(\chi_z\big|_{\Omega_\lambda}\big)_{z \in Z(\Omega_\lambda)} \Big)
\leq F_\lambda \Big( \big(\chi_z\big)_{z \in Z(\Omega_\lambda)} \Big)
= 3 \, \c6 |\tilde{\Omega}_\lambda|,
\end{equation}
where $\tilde{\Omega}_\lambda := \bigcup_{z \in Z(\Omega_\lambda)} H_z$.
Let $J_d(\partial \Omega_\lambda) := \partial\Omega_\lambda + B(0,d)$ denote the open $d$-neighborhood of $\partial \Omega_\lambda$. Since $\tilde \Omega_\lambda \subset \Omega_\lambda \cup J_d(\partial\Omega_\lambda)$, we can bound
\begin{equation}
|\tilde{\Omega}_\lambda|  \le |\Omega_\lambda| + |J_d(\partial \Omega_\lambda)|
 = V_\lambda + V_\lambda |J_\rho(\partial \Omega)|
\end{equation}
where $\rho = V_\lambda^{-\frac{1}{2}} d$.
Using~\eqref{char:Mink}, given $\eta>0$, we can find $\lambda_0 >0$ such that the following holds for all $0<\lambda < \lambda_0$:
\begin{equation}
\begin{aligned}
\label{UB2}
|\tilde{\Omega}_\lambda| & \le V_\lambda + V_\lambda (1+\eta) \mathcal{H}^1(\partial \Omega) 2 V_\lambda^{-\frac{1}{2}} d \\
& = V_\lambda + 2 d (1+\eta) \mathcal{H}^1(\partial \Omega_\lambda).
\end{aligned}
\end{equation}
Combining \eqref{UB1} and \eqref{UB2} completes the proof.
\end{proof}
\paragraph{Remark} Lemma \ref{UB} holds also if $\partial \Omega$ is $\mathcal{H}^1$-rectifiable and satisfies a density lower bound~\cite[Thm.~2.104]{AmbrosioFuscoPallara00}.

\bigskip
We conclude with the proof of Proposition~\ref{unconstr}.

\begin{proof}[Proposition~\ref{unconstr}, the dual formulation for $S^p$]
Assume first that $2<p<\infty$.
For $h \in \mathcal A^p(\mu)$ and $2<p<\infty$ one obtains
\begin{align}
\nonumber
\int_{\Omega} \phi \, \dd \mu & = -\int_{\Omega} \phi\, \Delta_p h  \, \dd x =
\int_{\Omega} |\nabla h|^{p-2} \nabla h \nabla \phi \, \dd x \\
\nonumber
& \le \frac{1}{p'} \int_{\Omega} |\nabla h|^p \, \dd x + \frac{1}{p} \int_{\Omega} |\nabla \phi |^p \, \dd x,
\end{align}
and thus $\Gamma_\Omega^p(\mu) \leq  S^p_\Omega(h)$.

On the other hand, if $\mu(\Omega)=|\Omega|$ and $\phi_{\max}$ is a maximizer satisfying the condition $\int_{\Omega} \phi_{\max}\, \dd x=0$, then there exists a Lagrange multiplier $\lambda \in \R$ such that the Euler-Lagrange equations
\begin{align*}
-\Delta_p\phi_{\max}  & =  \mu-\lambda \quad \text{ in } \Omega,\\
\frac{\partial}{\partial \nu} \phi_{\max}  & = 0 \quad \qquad \text{ on } \partial \Omega,
\end{align*}
are satisfied. Integration over $\Omega$ shows that $\lambda=1$ and thus $\phi_{\max} \in \mathcal A^p$.
Furthermore,
\begin{align*}
\int_{\Omega}|\nabla \phi_{\max}|^p \, \dd x =- \int_{\Omega} \phi_{\max}\,\Delta_p
\phi_{\max}\,\dd x
= \int_{\Omega} (\mu-1) \phi_{\max}\, \dd x= \int_{\Omega} \phi_{\max}\, \dd \mu,
\end{align*}
and therefore $S^p_\Omega(\phi_{\max}) = \Gamma_\Omega^p(\mu)$. This establishes (\ref{equival}) if $p \in (2,\infty)$.

The case $p = \infty$
follows immediately from the fact that
\[
\min_{h \in \mathcal{A}^\infty(\mu)} S^\infty_\Omega(h) = W_1(\mathcal{L}_\Omega,\mu),
\]
where $W_1$ is the $1$--Wasserstein transport cost, and that $W_1(\mathcal{L}_\Omega,\mu) = \Gamma^\infty_\Omega(\mu)$
by the Kantorovich-Rubinstein Theorem (see \cite[p.~34, Thm.~1.14]{Villani03}).
\end{proof}

%
%
%
%

\paragraph{Acknowledgements}
The majority of the work of D.~P. Bourne was carried out while he held a postdoc position at the
Technische Universiteit Eindhoven, supported by the grant `Singular-limit Analysis of Metapatterns', NWO grant 613.000.810.
Figure \ref{fig:NumericsSmallLambda} was produced in collaboration with Steven Roper.

%
%

\bibliographystyle{plain}
\bibliography{refsBournePeletierTheil}

\end{document}